\documentclass[british,11pt,a4paper]{amsart}

\usepackage{microtype,amssymb,hyperref}
\usepackage[T1]{fontenc}
\usepackage[marginratio=1:1]{geometry}
\usepackage[capitalize]{cleveref}
\usepackage[dvipsnames]{xcolor}

\numberwithin{equation}{section}

\renewcommand\frak{\mathfrak}
\renewcommand\Bbb{\mathbb}
\newcommand\Cal{\mathcal}
\newcommand{\x}{\times}
\renewcommand{\o}{\circ}

\newcommand{\al}{\alpha}
\newcommand{\be}{\beta}

\newcommand{\om}{\omega}

\renewcommand{\th}{\theta}
\newcommand{\si}{\sigma}

\newcommand{\La}{\Lambda}

\newcommand{\deq}{:=}
\newcommand{\Om}{\Omega}

\newcommand{\B}[2]{\langle#1,#2\rangle}
\newcommand{\CS}{\operatorname{CS}}

\newcommand{\tfg}{\widetilde{\mathfrak g}}

\newcommand{\Ad}{\operatorname{Ad}}
\newcommand{\tr}{\operatorname{tr}}
\newcommand{\Z}{\mathbb{Z}}
\renewcommand*{\d}{\mathrm{d}}
\newcommand*{\R}{\mathbb{R}}

\renewcommand{\i}{\mathrm{i}}
\newcommand{\col}{:}

\newcounter{theorem}
\numberwithin{theorem}{section}
\newtheorem{thm}[theorem]{Theorem}

\newtheorem{lemma}[theorem]{Lemma}
\newtheorem{prop}[theorem]{Proposition}
\newtheorem{cor}[theorem]{Corollary}

\theoremstyle{definition}
\newtheorem*{defNON}{Definition}
\newtheorem{definition}[theorem]{Definition}

\theoremstyle{remark}
\newtheorem{remark}[theorem]{Remark}

\title[Flat extensions and the Chern--Simons $3$-form]{Flat extensions of principal connections and the Chern--Simons $\boldsymbol{3}$-form}
\date{July 4, 2025}
\author[A.~{\v{C}}ap]{Andreas {\v{C}}ap}
\address{A.{\v{C}}. -- Faculty of Mathematics, University of Vienna, Vienna, Austria}
\email{Andreas.Cap@univie.ac.at}
\author[K.~J.~Flood]{Keegan J.~Flood}
\address{K.J.F. -- Faculty of Mathematics and Computer Science, UniDistance Suisse, Brig, Switzerland}
\email{keegan.flood@unidistance.ch}
\author[T.~Mettler]{Thomas Mettler}
\address{T.M. -- Faculty of Mathematics and Computer Science, UniDistance Suisse, Brig, Switzerland}
\email{thomas.mettler@fernuni.ch, mettler@math.ch}

\begin{document}

\begin{abstract}

We introduce the notion of a \emph{flat extension} of a connection $\theta$ on a principal bundle. Roughly speaking, $\theta$ admits a flat extension if it arises as the pull-back of a component of a Maurer--Cartan form. For trivial bundles over closed oriented $3$-manifolds, we relate the existence of certain flat extensions to the vanishing of the Chern-–Simons invariant associated with $\theta$. As an application, we recover the obstruction of Chern--Simons for the existence of a conformal immersion of a Riemannian $3$-manifold into Euclidean $4$-space. In addition, we obtain corresponding statements for a Lorentzian $3$-manifold, as well as a global obstruction for the existence of an equiaffine immersion into $\R^4$ of a $3$-manifold that is equipped with a torsion-free connection preserving a volume form.

\end{abstract}

\maketitle

\section{Introduction}

Chern--Simons forms and invariants derived from them are prominent examples of secondary invariants of certain types of connections \cite{MR0353327}. Apart from their interest in geometry and topology, Chern--Simons forms also play a fundamental role in theoretical physics. In contrast to standard characteristic classes and characteristic numbers which are defined on even-dimensional manifolds, Chern--Simons invariants are defined for manifolds of odd dimension. Among Chern--Simons forms, much interest is devoted to the Chern--Simons $3$-form which is defined for a $1$-form $\theta$ with values in the Lie algebra $\mathfrak{g}$ of a Lie group $G$ via
\[
\CS(\th)=\B{\th}{\d\th}+\tfrac13 \B{\th}{[\th,\th]},
\]
where $\B{\,\cdot\,}{\cdot\,}$ denotes a symmetric bilinear form on $\mathfrak{g}$ which is invariant under the adjoint action $\Ad$ of $G$ and $[\,\cdot\,,\cdot\,]$ denotes the Lie bracket of $\mathfrak{g}$. In the case where $\theta$ is a principal connection on a trivial principal $G$-bundle $\pi : P \to M$ over a closed oriented $3$-manifold $M$, the Chern--Simons $3$-form can be used to assign a real number
\[
c_{\sigma}=\int_M \sigma^*\CS(\theta)
\] 
to every global smooth section $\sigma : M \to P$. Depending on the topology of $G$, it may happen that $c_\sigma$ is independent of $\si$ or that one can choose $\B{\,\cdot\,}{\cdot\,}$ so that $c_{\sigma}$ is independent of $\sigma$ up to addition of an integer. In the former case, one obtains a real valued Chern--Simons invariant associated to $\theta$, in the latter -- more frequent -- case, an invariant with values in $\R/\Z$. 

The prototypical example of an $\R/\Z$-valued Chern--Simons invariant arises from considering the $\mathrm{SO}(3)$-bundle $\pi : \mathcal{SO}M \to M$ of orientation compatible orthonormal frames of a closed oriented Riemannian $3$-manifold $(M,g)$, equipped with its Levi-Civita connection form $\theta$ and where $\B{\,\cdot\,}{\cdot\,}$ is a suitable scalar multiple of the Killing form of $\mathfrak{so}(3)$. In \cite{MR0353327}, Chern--Simons show that the Riemannian $3$-manifold $(M,g)$ can be isometrically immersed into Euclidean $4$-space $\mathbb{E}^4$ only if $c_{\sigma}$ is an integer. Moreover, they observe that the associated $\R/\Z$-valued invariant is actually conformally invariant, so that the integrality of $c_{\sigma}$ is an obstruction to the existence of a conformal immersion into $\mathbb{E}^4$. The Chern--Simons invariant of a Riemannian $3$-manifold has since played an important role in hyperbolic geometry due to its relation to the $\eta$-invariant \cite{MR807069}. See also \cite{MR3228423} and \cite{MR3159164} for more recent related work.

In $3$-dimensional CR-geometry, the Chern--Simons $3$-form gives rise to an
$\R$-valued Chern--Simons invariant as discovered by Burns--Epstein
\cite{MR0936085}. To include this and similar examples, we work in the setting of
$\frak g$-connections on principal bundles whose structure group is a subgroup of
$G$.

For differential geometric structures not related to Riemannian or CR-geometry, the meaning of the Chern--Simons invariant seems to have received less attention in the literature. In this article, we relate vanishing statements for the Chern--Simons invariant to the notion of a \emph{flat extension of a principal connection}. To this end, suppose that $G$ is a Lie subgroup of a Lie group $\tilde{G}$ which is equipped with an $\Ad$-$\tilde{G}$ invariant bilinear form $\B{\,\cdot\,}{\cdot\,}$ on its Lie algebra $\tilde{\mathfrak{g}}$. We assume that the restriction of $\B{\,\cdot\,}{\cdot\,}$ to $\mathfrak{g}\times \mathfrak{g}$ is non-degenerate so that $\tilde{\mathfrak{g}}=\mathfrak{g}\oplus \mathfrak{g}^{\perp}$, where $\mathfrak{g}^{\perp}$ denotes the orthogonal complement of $\mathfrak{g}$ with respect to $\B{\,\cdot\,}{\cdot\,}$. Writing a $\tilde{\mathfrak{g}}$-valued $1$-form $\psi$ as $\psi=\psi^{\top}+\psi^{\perp}$ with $\psi^{\top}$ taking values in $\mathfrak{g}$ and $\psi^{\perp}$ taking values in $\mathfrak{g}^{\perp}$, we define:
\begin{definition}\label{defn:flatext}
Let $P \to M$ be a principal $G$-bundle and $\theta \in \Omega^1(P,\mathfrak{g})$ a connection. A \emph{flat extension of $\theta$ of type $(\tilde{G},G)$} is a bundle homomorphism $F : P \to \tilde{G}$ into the total space of the principal $G$-bundle $\tilde{G} \to \tilde{G}/G$ so that
\[
\theta=F^*(\mu_{\tilde{G}}^{\top}),
\]
where $\mu_{\tilde{G}}$ denotes the Maurer--Cartan form of $\tilde{G}$. 
\end{definition} 
Let $\CS(\mu_{\tilde G})$ denote the Chern--Simons form of the Maurer--Cartan form $\mu_{\tilde G}$ of $\tilde{G}$, computed with respect to $\B{\,\cdot\,}{\cdot\,}$. Moreover, we compute the Chern--Simons $3$-form of a $\mathfrak{g}$-valued $1$-form with respect to the bilinear form on $\mathfrak{g}$ obtained by restricting the bilinear form $\B{\,\cdot\,}{\cdot\,}$ on $\tilde{\mathfrak{g}}$ to $\mathfrak{g}\times \mathfrak{g}$. In the case where the pair $(\tilde{\mathfrak{g}},\mathfrak{g})$ of Lie algebras is a \emph{symmetric pair} (see \cref{sect:partialblindness} for details), we obtain:
\begin{cor}\label{cor:keystatement}
Suppose $P \to M$ is a trivial principal $G$-bundle over a closed oriented $3$-manifold $M$ and $\theta \in \Omega^1(P,\mathfrak{g})$ a connection admitting a flat extension of type $(\tilde{G},G)$ with $(\tilde{\mathfrak{g}},\mathfrak{g})$ being a symmetric pair. If $\CS(\mu_{\tilde{G}})$ is exact, then $\int_M\sigma^*\CS(\theta)=0$ and if $\CS(\mu_{\tilde{G}})$ represents an element of $H^3(\tilde{G},\Z)$, then $\int_M\sigma^*\CS(\theta) \in \Z$ for every global smooth section $\sigma : M \to P$.
\end{cor}

\cref{cor:keystatement} follows from our slightly more general main \cref{thm:main-flat}. The proof of this theorem strongly depends on a certain algebraic identity -- see \cref{lem:CS-pairs} -- which expresses $\mathrm{CS}(\psi)$ as the sum of $\mathrm{CS}(\psi^{\top})$ and a term involving the curvature of $\psi$. In particular, if $\psi$ satisfies the Maurer--Cartan equation $0=\d\psi+\tfrac{1}{2}[\psi,\psi]$, then $\CS(\psi)=\CS(\psi^{\top})$, so that the Chern--Simons $3$-form only detects the $\mathfrak{g}$-valued part of $\psi$.  

The obstruction of Chern--Simons is obtained from our main result by observing that the Gauss map of an isometric immersion $(M,g) \to \mathbb{E}^4$ is covered by a map $F : \mathcal{SO}M \to \mathrm{SO}(4)$ which is a flat extension of the Levi-Civita connection $\theta$ of type $(\mathrm{SO}(4),\mathrm{SO}(3))$. Since $(\mathfrak{so}(4),\mathfrak{so}(3))$ is a symmetric pair, it remains to argue that for a suitable choice of $\B{\,\cdot\,}{\cdot\,}$ on $\mathfrak{so}(4)$, the $3$-form $\CS(\mu_{\mathrm{SO}(4)})$ represents an element of $H^3(\mathrm{SO}(4),\Z)$. This requires a good understanding of the third integral homology group of $\mathrm{SO}(4)$. The relevant calculations are carried out in \cref{Appendix}. 

As a further application  -- see \cref{prop:Lorentz-flat-ext} --  we discuss the case of space- and time-oriented Lorentzian $3$-manifolds that admit a global orthonormal frame. Here the Chern--Simons invariant is $\R$-valued and for appropriate choices of $\tfg$ flat extensions are obtained from an isometric immersion into the Lorentzian vector space $\Bbb R^{3,1}$ and the split-signature vector space $\Bbb R^{2,2}$, respectively. Existence of such an immersion then leads to integrality, respectively, vanishing of the Chern--Simons invariant. 

The group $\mathrm{SO}(3)\subset \mathrm{SL}(3,\R)$ is a strong deformation retract by Iwasawa decomposition. As a consequence of this, one also obtains an $\R/\Z$-valued Chern--Simons invariant for an oriented $3$-manifold $M$ equipped with a torsion-free connection $\nabla$ on its tangent bundle that preserves some volume form $\nu$. In \cref{thm:equiaffine} we show that the vanishing of this invariant obstructs the existence of an equiaffine immersion of $(M,\nabla,\nu)$ into $\R^4$, where $\R^4$ is equipped with its standard flat connection and volume form.   

In \cref{app:examples} we discuss two examples. In particular, applying \cref{thm:equiaffine}, we show that for $\mathbb{RP}^3$ equipped with the Levi-Civita connection $\nabla$ and volume form $\nu$ arising from its standard metric, there exists no global equiaffine immersion into $\R^4$. 

\subsection*{Acknowledgments} T.M.~is grateful to Lukas Lewark for helpful communications.

\section{Basics on the Chern--Simons $3$-form}\label{1}

We start by collecting some basic facts about Chern--Simons $3$-forms. We refer to
\cite{MR1337109} for additional context. We consider a Lie algebra $\mathfrak{g}$ and study
$\frak{g}$-valued differential forms. In addition, we assume that $\frak{g}$ is endowed
with a non-degenerate $\frak{g}$-invariant, symmetric bilinear form, which we denote
by $\B{\,\cdot\,}{\cdot\,}$. Recall that for a simple Lie algebra $\frak{g}$ any such form has
to be a multiple of the Killing form, so there is a unique such form up to scale in
this case. The normalization of the form will be important in what follows, 
however. For a manifold $N$, we will denote by $\Om^k(N,\frak{g})$ the space of $\frak
g$-valued $k$-forms on $N$. The main case we are interested in is that $N$ is
the total space of a principal fibre bundle and we aim at invariants defined on the
base $M$ of that bundle. For any $N$, we have two basic operations on $\frak{g}$-valued
forms, as
\begin{gather*}
\B{\,\cdot\,}{\cdot\,}\col\Om^p(N,\frak{g})\x\Om^q(N,\frak{g})\to \Om^{p+q}(N) \\
[\,\,,\,]\col\Om^p(N,\frak{g})\x\Om^q(N,\frak{g})\to \Om^{p+q}(N,\frak{g}).
\end{gather*}
These are characterized by the fact that for $\al\in\Om^p(N)$, $\be\in\Om^q(N)$ and $X,Y\in\frak{g}$, we have 
\begin{equation}\label{eq:basic-ops}
\B{\al\otimes X}{\be\otimes Y}\deq\B{X}{Y} \al\wedge\be, \qquad
[\al\otimes X,\be\otimes Y]\deq\al\wedge\be\otimes[X,Y]. 
\end{equation}
The definitions immediately imply that for $\om\in\Om^p(N,\frak{g})$ and
$\theta\in\Om^q(N,\frak{g})$ one obtains
\[
\B{\om}{\theta}=(-1)^{pq}\B{\theta}{\omega}, \qquad
[\om,\theta]=(-1)^{pq+1}[\theta,\om], 
\]
as well as
\begin{gather}
\label{eq:dB} \d\B{\om}{\theta}=\B{\d\om}{\th}+(-1)^p\B{\om}{\d\th}, \\
\label{eq:dL}  \d[\om,\th]=[\d\om,\th]+(-1)^p[\om,\d\th].
\end{gather}
The fact that $\B{\,\cdot\,}{\cdot\,}$ is $\frak{g}$-invariant respectively the Jacobi identity for $[\ ,\ ]$ implies that for an additional form
$\tau\in\Om^r(N,\frak{g})$, we obtain
\begin{gather}
\label{eq:B-invar}\B{\om}{[\tau,\th]}=\B{[\om}{\tau],\theta},\\
\label{eq:Jacobi}
[\om,[\tau,\theta]]=[[\om,\tau],\theta]+(-1)^{pr}[\tau,[\om,\theta]].
\end{gather}

\begin{definition}\label{def:CS}
For $\theta\in\Om^1(N,\frak{g})$, the \textit{Chern--Simons $3$-form} $\CS(\th)\in\Om^3(N)$ is defined as
\[
\CS(\th)\deq\B{\th}{\d\th}+\tfrac13 \B{\th}{[\th,\th]}. 
\]
\end{definition}
Notice that when we write $\Theta=\d\theta+\frac{1}{2}[\theta,\theta]$ we have
\[
\CS(\theta)=\B{\theta}{\Theta}-\tfrac{1}{6}\B{\theta}{[\theta,\theta]}.
\]
The following well-known property is one of the main motivations for the definition of
the Chern--Simons $3$-form: 
\begin{equation}\label{eqn:extdifcs}\d\CS(\th)=\B{\Theta}{\Theta}.\end{equation} 
The construction is natural in the sense that if $M$ is a smooth manifold
and $\sigma \col M \to N$ a smooth map, then
\begin{equation}\label{eqn:cspullback}
\sigma^*\CS(\theta)=\CS(\sigma^*\theta).
\end{equation}
Note that \eqref{eqn:extdifcs} in particular implies that if $\th$ satisfies the Maurer--Cartan equation $\d\th+\frac12[\th,\th]=0$, then $\d\CS(\th)=0$ and hence $\CS(\th)$ determines a well-defined cohomology class in $H^3(N,\R)$.

\section{Chern--Simons invariants}\label{1.3}

The standard setting for Chern--Simons invariants associated to $(\frak g,\B{\,}{\,})$ uses principal connection forms on principal fibre bundles with structure group a Lie group $G$ with Lie algebra $\frak{g}$. In some situations, it is advantageous to use a more general setting, namely to start from principal bundles whose structure group $H$ is a Lie subgroup of a Lie group $G$ with Lie algebra $\frak g$. Then the Lie algebra $\frak{h}$ of $H$ naturally is a subalgebra of $\frak{g}$ and we can then restrict the adjoint representation of $G$ to $H$. This provides an extension of the adjoint representation of $H$ to a representation on $\frak{g}$ that we also denote by $\Ad$ if there is no risk of confusion.

In this situation, there is a natural notion of $\frak{g}$-connections on principal
fibre bundles with structure group $H$. To formulate this, assume that $\pi\col
P\to M$ is a principal $H$-bundle. Let $R \col P \times H \to P$ denote the
principal right action and define for all $h \in H$ the map $R_h=R(\cdot,h) \col P
\to P$ and for all $u \in P$ the map
$\iota_u=R(u,\cdot) \col H \to P$.

\begin{definition}\label{defn:gconnection}
Let $G$ be a Lie group with Lie algebra $\mathfrak{g}$ and $H\subset G$ a Lie
subgroup. A $\mathfrak{g}$-valued $1$-form $\theta \in \Omega^1(P,\mathfrak{g})$ on
some principal $H$-bundle $\pi \col P \to M$ is called a
\emph{$\mathfrak{g}$-connection} if for all $u\in P$ and all $h \in H$ we
have
\begin{equation}\label{eqn:gvalued1form}
\iota_u^*\theta=\mu_H \qquad \text{and} \qquad R_h^*\theta=\Ad(h^{-1})\circ \theta,
\end{equation}
where $\mu_H$ denotes the Maurer--Cartan form of $H$. 
\end{definition}

\begin{remark}\leavevmode
\begin{enumerate}
\item[(i)] Notice that if $H=G$, then the definition of $\mathfrak{g}$-connection agrees with the standard notion of a principal connection.
\item[(ii)] The first condition in the definition is equivalent to the fact that $\theta$ reproduces the generators of fundamental vector fields, i.e.\ that $\theta(\zeta_A)=A$ for any $A\in\frak{h}$. Here $\zeta_A(u)=\tfrac{\d}{\d t}|_{t=0}R(u,\exp(tA))$. 
\item[(iii)] Let $\pi\col P\to M$ be a principal $H$-bundle and $\theta\in\Omega^1(P,\frak{g})$ a $\mathfrak{g}$-connection. Then we can extend the structure group to $G$ by forming $\hat{P}\col=P\x_H G\to M$ and there is a canonical inclusion $i\col P\to \hat{P}$. One easily shows that there is a unique principal connection $\hat{\theta}\in\Omega^1(\hat{P},\frak{g})$ such that $i^*\hat{\theta}=\theta$.
\item[(iv)] Conversely, suppose $\hat{\pi} \col \hat{P} \to M$ is a principal $G$-bundle and $\hat{\theta} \in \Omega^1(\hat{P},\mathfrak{g})$ is a principal connection. If $\pi \col P \to M$ is a reduction of $\hat{\pi} : \hat{P} \to M$ to the structure group $H$, then restricting $\hat{\theta}$ to $P$ yields a $\mathfrak{g}$-connection on $\pi : P \to M$. 
\item[(v)] An important source of $\mathfrak{g}$-connections with $H\neq G$ is provided by canonical Cartan connections associated to geometric structures. We will discuss examples of this in \cref{1.4} and study Chern--Simons invariants in the context of Cartan connections in more detail in a forthcoming article.
\end{enumerate}
\end{remark}

Fixing $(\frak{g},\B{\,\cdot\,}{\cdot\,})$ and $H\subset G$, we can consider the Maurer--Cartan
form $\mu_H\in\Om^1(H,\frak{h})$ and we can of course also view $\mu_H$ as an element
of $\Om^1(H,\frak{g})$. Hence we can form the associated Chern--Simons form
$\CS(\mu_H)\in\Om^3(H,\R)$, which we will also denote by
$\CS^{\frak{g}}(\mu_H)$ to emphasize the role of $\frak{g}$. Since $\mu_H$ satisfies
the Maurer--Cartan equation, this form is closed by \eqref{eqn:extdifcs} and hence
determines a cohomology class $[\CS^{\frak{g}}(\mu_H)]\in H^3(H,\R)$.

Explicitly, $\CS(\mu_H)$ is the left-invariant $3$-form on $H$, which is induced by
the trilinear map
\[
\frak{h}^3\to\R, \qquad (X,Y,Z)\mapsto
-\frac16\B{X}{[Y,Z]}
\]
which is complete alternating by invariance of $\B{\,\cdot\,}{\cdot\,}$. If $\frak{h}$ is simple,
then the restriction of $\B{\,\cdot\,}{\cdot\,}$ to $\frak{h}$ has to be a multiple
of the Killing form, and hence $\CS^{\frak{g}}(\mu_H)$ is a multiple of the so-called
\textit{Cartan $3$-form} on $H$.
It is well-known that for a compact simple Lie group $H$, the cohomology class of the Cartan $3$-form  spans
  $H^3(H,\R)\cong\R$. We now have:
  \begin{prop}\label{prop:pullbacks}
  Let $M$ be a closed oriented $3$-manifold, $\pi \col P \to M$ be a principal
  $H$-bundle that admits a global smooth section and let $\theta\in\Om^1(P,\frak{g})$
  be a $\mathfrak{g}$-connection. For a smooth section $\si\col M\to P$
  consider 
  \[
  c_\si\deq\int_M \si^*\CS(\theta)\in\R.
  \]
  \begin{enumerate}
    \item[(i)] If $\CS^{\frak{g}}(\mu_H)$ is exact, then $c_\si$ is independent of $\si$
    and hence defines an invariant of the form $\theta$.

    \item[(ii)] Suppose that $\B{\,\cdot\,}{\cdot\,}$ is chosen in such a way that $\CS^{\frak
    g}(\mu_H)$ represents an element of $H^3(H,\Z)$. Then $c_\si+\Z\in\R/\Bbb
    Z$ is independent of $\si$, and hence we obtain an invariant of $\theta$ with
    values in $\R/\Z$.
  \end{enumerate}
  \end{prop}

  The proof of \cref{prop:pullbacks} relies on the identity \eqref{eqn:keyid} below
  which is standard in the case of principal connections (see for instance
  \cite{MR1337109}). The case of $\mathfrak{g}$-valued connection forms is quite
  similar. For the convenience of the reader we include a proof in our more
  general setting: 
  \begin{lemma} Fix $(\frak{g},\B{\,\cdot\,}{\cdot\,})$ and $H\subset G$ as above.
  Then for a $\mathfrak{g}$-connection $\theta \in \Omega^1(P,\mathfrak{g})$ on a
  principal $H$-bundle $\pi \col P \to M$ we have, not indicating explicitly
  pullbacks along projections in a product,
  \begin{equation}\label{eqn:keyid}
  R^*\CS(\theta)=\CS(\theta)+\CS^{\frak g}(\mu_H)+\d\B{\Ad^{-1}\circ\,\theta}{\mu_H},
  \end{equation}
  where $\Ad^{-1}\col=\Ad\circ I_H$ denotes the composition of the inversion $I_H \col H \to H$, $h \mapsto h^{-1}$ and the adjoint representation of $H$, thought of as acting on $\mathfrak{g}$.
  \end{lemma}
  \begin{proof}
  Throughout this proof, we continue to not indicate pullbacks of differential forms
  to a product along the projections explicitly. In this language,
  \eqref{eqn:gvalued1form} is equivalent to
  $R^*\theta=\mu_H+\Ad^{-1}\circ\,\theta$. Putting $\alpha\deq\Ad^{-1}\circ\,\theta$,
  we first claim that
  \begin{equation}\label{eqn:extderivad}
  \d\alpha=\Ad^{-1}\circ\,\d\theta-[\mu_H,\alpha].
  \end{equation}
  This can be proved by a direct computation. Alternatively, writing
  $\Theta=\d\theta+\frac{1}{2}[\theta,\theta]$, the equations
  \eqref{eqn:gvalued1form} imply that $i_u^*\Theta=0$ and
  $R_h^*\Theta=\Ad(h^{-1})\circ\Theta$, which reads as
  \begin{equation}\label{eqn:rightactioncurv}
  R^*\Theta=\Ad^{-1}\circ\,\Theta.
  \end{equation}
  On the other hand, we can compute directly
  \[
  \begin{aligned}
  R^*\Theta&=R^*\left(\d\theta+\tfrac{1}{2}[\theta,\theta]\right)=\d(R^*\theta)+\tfrac{1}{2}[R^*\theta,R^*\theta]\\
  &=\d(\mu_H+\alpha)+\tfrac{1}{2}[\mu_H+\alpha,\mu_H+\alpha]\\
  &=\d\alpha+[\mu_H,\alpha]+\tfrac{1}{2}\Ad^{-1}\circ [\theta,\theta],
\end{aligned}
\]
where the fourth equality uses $\d\mu_H+\frac{1}{2}[\mu_H,\mu_H]=0$. 
Comparing this to \eqref{eqn:rightactioncurv} proves \eqref{eqn:extderivad}.
Now we compute
\[
\begin{aligned}
\B{R^*\theta}{\d(R^*\theta)}&=\B{\alpha}{\d\alpha}+\B{\alpha}{\d\mu_H}+\B{\mu_H}{\d \alpha}+\B{\mu_H}{\d \mu_H}\\
&=\B{\alpha}{\d\alpha}+\B{\mu_H}{\d \mu_H}+\d\B{\alpha}{\mu_H}+2\B{\alpha}{\d\mu_H}\\
&=\B{\alpha}{\d\alpha}+\B{\mu_H}{\d \mu_H}+\d\B{\alpha}{\mu_H}-\B{\alpha}{[\mu_H,\mu_H]},
\end{aligned}
\]
where the second equality uses \eqref{eq:dB} and the third that $\d\mu_H+\frac{1}{2}[\mu_H,\mu_H]=0$. Using \eqref{eq:B-invar} we also obtain
\[
\B{R^*\theta}{[R^*\theta,R^*\theta]}=\B{\alpha}{[\alpha,\alpha]}+\B{\mu_H}{[\mu_H,\mu_H]}+3\B{\alpha}{[\mu_H,\mu_H]}+3\B{\mu_H}{[\alpha,\alpha]}.
\]
In summary, we thus have
\[
R^*\CS(\theta)=\CS(R^*\theta)=\B{\alpha}{\d\alpha}+\tfrac{1}{3}\B{\alpha}{[\alpha,\alpha]}+
\B{\mu_H}{[\alpha,\alpha]}+\CS(\mu_H)+\d\B{\alpha}{\mu_H}. 
\]
Since the last term already shows up in \eqref{eqn:keyid}, it remains to show that 
\[
\CS(\theta)=\B{\alpha}{\d \alpha}+\tfrac{1}{3}\B{\alpha}{[\alpha,\alpha]}+\B{\mu_H}{[\alpha,\alpha]}.
\] 
Using \eqref{eqn:extderivad} and the $\Ad$-invariance of $\B{\,\cdot\,}{\cdot\,}$ gives
\[
\B{\alpha}{\d\alpha}=-\B{\alpha}{[\mu_H,\alpha]}+\B{\Ad^{-1}\circ\, \theta}{\Ad^{-1}\circ\, \d\theta}=-\B{\alpha}{[\mu_H,\alpha]}+\B{\theta}{\d\theta}.
\]
With $[\Ad(h)(x),\Ad(h)(y)]=\Ad(h)([x,y])$ it follows that
\[
\B{\alpha}{[\alpha,\alpha]}=\B{\Ad^{-1}\circ\, \theta}{\Ad^{-1}\circ\, [\theta,\theta]}=\B{\theta}{[\theta,\theta]}
\]
and we conclude that
\[
\B{\alpha}{\d \alpha}+\tfrac{1}{3}\B{\alpha}{[\alpha,\alpha]}+\B{\mu_H}{[\alpha,\alpha]}=\B{\theta}{\d\theta}+\tfrac{1}{3}\B{\theta}{[\theta,\theta]}=\CS(\theta),
\]
which finishes the proof. 
\end{proof}
\begin{proof}[Proof of \cref{prop:pullbacks}]
Given one global smooth section $\si\col M\to N$, any other global smooth section is of the form $\hat{\sigma}=R\circ (\sigma,h) \col M \to N$ for some smooth map $h \col M \to H$. Using \eqref{eqn:cspullback} and \eqref{eqn:keyid} we thus obtain
\[
\begin{aligned}
\hat{\sigma}^*\CS(\theta)&=(\sigma,h)^*\left(R^*\CS(\theta)\right)\\
&=\CS(\sigma^*\theta)+h^*\CS^{\frak g}(\mu_H)+\d\B{\Ad(h^{-1})\circ(\sigma^*\theta)}{h^*\mu_H},
\end{aligned}
\]
so that integration yields
\begin{equation}\label{eqn:chernsimonschangesection}
c_{\hat{\sigma}}=c_{\sigma}+\int_M h^*\CS^{\frak g}(\mu_H)
\end{equation}
by Stokes' theorem. The claims follow immediately from \eqref{eqn:chernsimonschangesection}. 
\end{proof}

\begin{remark}
Observe that this result does not depend on any conditions on the restriction of
$\B{\,\cdot\,}{\cdot\,}$ to $\frak{h}$, so no assumptions in that direction are
needed in order to get well-defined Chern--Simons invariants. Indeed, for several of
the examples discussed below, the restriction is degenerate.
\end{remark}

\section{Examples of Chern--Simons invariants}\label{1.4}
Recall that any orientable $3$-manifold is parallelizable and hence admits a global smooth frame for the tangent bundle.

(1) Taking $H=G=\mathrm{SO}(3)$, this recovers the original definition of Chern--Simons by using the Levi-Civita connection on oriented Riemannian $3$-manifolds. Starting from a global frame of the tangent bundle $TM$, we can apply Gram-Schmidt to obtain a global orthonormal frame, which shows that the orthonormal frame bundle admits global smooth sections. Since $H^3(G,\R)\cong\R$ we immediately conclude from \cref{prop:pullbacks} that we obtain an invariant with values in $\R/\Z$ provided that we normalize $\B{\,\cdot\,}{\cdot\,}$ in such a way that $\int_{\mathrm{SO}(3)}\mu_{\mathrm{SO}(3)}\in\Z$ and the best choice is to ensure that $\int_{\mathrm{SO}(3)}\mu_{\mathrm{SO}(3)}=\pm 1$. Such a normalization is computed explicitly in \cref{Appendix}.
  
\smallskip

(2) Take $H=G=\mathrm{SO}_0(2,1)$ and the Levi-Civita connection on a space- and time-oriented Lorentzian $3$-manifold. In this case, we have to assume in addition that there is a global orthonormal frame for the given Lorentzian metric, i.e.\ that the orthonormal frame bundle admits a global smooth section. Since the maximal compact subgroup of $G$ is contained in $\mathrm{S}(\mathrm{O}(2)\x \mathrm{O}(1))$, we get $H^3(G,\R)=\{0\}$ in this case. Hence regardless of the normalization of $\B{\,\cdot\,}{\cdot\,}$, we get an $\R$-valued invariant (assuming existence of a global orthonormal frame). As we shall see below, the normalization of $\B{\,\cdot\,}{\cdot\,}$ still can be relevant here, since there are integrality results for the invariant on manifolds that admit certain isometric immersions, see \cref{prop:Lorentz-flat-ext} below.

\smallskip

(3) Take $H=G=\mathrm{SL}(3,\R)$ and consider oriented $3$-manifolds $M$ endowed with  a fixed volume form $\nu$ and a linear connection $\nabla$ on $TM$ preserving $\nu$. Then we can apply our construction to the principal connection induced by $\nabla$ on the $\mathrm{SL}(3,\R)$-frame bundle of $M$ defined by $\nu$. The inclusion $\mathrm{SO}(3)\to \mathrm{SL}(3,\R)$ induces an isomorphism in cohomology, so $H^3(H,\R)\cong\Z$. Choosing $\B{\,\cdot\,}{\cdot\,}$ appropriately, we obtain an $\R/\Z$-valued invariant for such connections.

\smallskip

(4) An example with $H\neq G$ is provided by the Burns-Epstein invariant introduced in \cite{MR0936085}. Here the setting is that $M$ is a compact oriented $3$-manifold endowed with a CR-structure, i.e.\ a contact distribution $C\subset TM$ which is endowed with an almost complex structure. It is a classical result due to E.\ Cartan, see \cite{MR1556687}, that $M$ admits a canonical $\frak{g}$-connection with $\frak{g}=\frak{su}(2,1)$ on a principal fibre bundle constructed from the CR structure. Indeed, this is a Cartan connection in modern terminology.

The structure group $H$ of the canonical principal bundle is a (parabolic) subgroup of $G=\mathrm{PSU}(2,1)$. The details on $H$ are not very important here, it comes from the stabilizer in $\mathrm{SU}(2,1)$ of a null line in $\mathbb C^3$. It turns out that $H$ is isomorphic to a semi-direct product of $\mathrm{U}(1)$ and the complex Heisenberg group of real-dimension $3$, which in particular implies that $H^3(H,\R)=\{0\}$.

To obtain an invariant, one has to assume that the $H$-principal bundle associated to the CR structure admits a global section, which by orientability of $M$ turns out to be equivalent to triviality of the CR subbundle $C\subset TM$. Equivalently, this can be expressed as existence of a global CR vector field on $M$. Since $\frak{g}$ is simple, the form $\B{\,\cdot\,}{\cdot\,}$ has to be a non-zero multiple of the Killing form, and for each choice of such a form, our construction leads to a real-valued CR invariant. Notice that in this case the restriction of $\B{\,\cdot\,}{\cdot\,}$ to $\frak{h}$ is degenerate with null-space the nilradical of $\frak{h}$.

\smallskip

(5) The construction of canonical principal bundles and (Cartan) connections from (4) is a special case of the general constructions for parabolic geometries, see \cite{MR2532439}. In particular, there are two more cases of structures on $3$-manifolds that have an underlying contact structure. One of those are Legendrian contact structures for which the additional ingredient to a contact distribution $C\subset TM$ is a decomposition $C=E\oplus F$ as the direct sum of two  line subbundles (which are automatically Legendrian). Here $\frak{g}=\frak{sl}(3,\R)$ and $H$ comes from the subgroup of upper triangular matrices in $\mathrm{SL}(3,\R)$, so $H^3(H,\R)=\{0\}$. To obtain an invariant, one again has to assume that $M$ is compact and oriented and that the Legendrian subbundles $E$ and $F$ are trivial, and then one obtains an $\R$-valued invariant.

The other example are so-called contact projective structures, see also \cite{MR2189678}. Here the additional structure is given by a family of curves tangent to the contact distribution $C\subset TM$ with one curve through each point in each direction (in $C$), which can be realized as geodesics of a linear connection. For this example, $\frak{g}=\frak{sp}(4,\R)$ and $H$ comes from a subgroup of $\mathrm{Sp}(4,\R)$ which is a semi-direct product of $\mathrm{Sp}(2,\R)\cong \mathrm{SL}(2,\R)$ with a $3$-dimensional real Heisenberg group. Hence $H^3(H,\R)=\{0\}$ and assuming triviality of the canonical principal bundle over a compact oriented $3$-manifold, one thus obtains a real-valued invariant for any choice of $\B{\,\cdot\,}{\cdot\,}$.

\section{Partial blindness for flat Lie algebra valued forms}\label{sect:partialblindness}

It is natural to try to find conditions on the pair $(\pi \col P \to M,\theta)$ which  imply vanishing of the associated Chern--Simons invariant. We will achieve this in \cref{sect:flatextensionGequlasH} below. Our result crucially relies on a certain algebraic feature of the Chern--Simons $3$-form that we will refer to as \emph{partial blindness for flat Lie algebra valued forms}.

Here our basic setup is that we study Chern--Simons invariants associated to $(\frak{g},\B{\,\cdot\,}{\cdot\,})$ via a realization of $\frak{g}$ as a Lie subalgebra of a bigger Lie algebra $\tilde{\mathfrak{g}}$ in such a way that $\B{\,\cdot\,}{\cdot\,}$ is the restriction of an invariant form on $\tilde{\mathfrak{g}}$ which we denote by the same symbol. In particular we \textit{do} assume that the restriction to $\frak{g}$ is non-degenerate here. (Note that this is automatically satisfied if $\mathfrak{g}$ is simple and the restriction is non-zero.)
Then we can form the orthogonal space $\mathfrak{g}^\perp\subset\mathfrak{\tilde{g}}$
which by our assumption is complementary to $\mathfrak{g}$, so
$\mathfrak{\tilde{g}}=\mathfrak{g}\oplus\mathfrak{g}^\perp$ as a vector space. But
since $\B{\,\cdot\,}{\cdot\,}$ is clearly $\mathfrak{g}$-invariant, also $\frak
g^\perp\subset\mathfrak{\tilde{g}}$ is a $\mathfrak{g}$-invariant subspace, so
$\mathfrak{\tilde{g}}=\frak{g}\oplus\mathfrak{g}^\perp$ as a representation of
$\mathfrak{g}$. In particular, this implies that
$[\mathfrak{g},\mathfrak{g}^\perp]\subset\mathfrak{g}^\perp$. An important special
case is that the complement $\mathfrak{g}^\perp$ makes
$(\mathfrak{\tilde{g}},\mathfrak{g})$ into a \emph{symmetric pair}, i.e., that in
addition we have $[\mathfrak{g}^\perp,\mathfrak{g}^\perp]\subset\mathfrak{g}$.

Now any $\mathfrak{\tilde{g}}$-valued form $\om\in\Om^k(N,\mathfrak{\tilde{g}})$ on a
manifold $N$ decomposes accordingly as
\[
\om=\om^\top+\om^\perp
\] 
with $\om^\top\in\Om^k(N,\mathfrak{g})$ and
$\om^{\perp}\in\Om^k(N,\mathfrak{g}^\perp)$.

We can now state a crucial technical lemma:

\begin{lemma}\label{lem:CS-pairs}
Let $\th\in\Om^1(N,\mathfrak{\tilde{g}})$ be such that for the decomposition
$\th=\th^\top+\th^\perp$, we have
$[\th^\perp,\th^\perp]\in\Om^2(N,\mathfrak{g})$. Then for
$\Theta=\d\th+\frac12[\th,\th]$ with associated decomposition
$\Theta=\Theta^\top+\Theta^\perp$, we obtain
\begin{equation}\label{eqn:partialblindness}
\CS(\th)=\CS(\th^\top)+\B{\th^\perp}{\Theta^\perp}. 
\end{equation}
\end{lemma}
\begin{proof}
We can decompose the equation $\d\th=\Theta-\frac12[\th,\th]$ into components, and by
our assumption on $[\th^\perp,\th^\perp]$, this reads as
\begin{equation}\label{eq:dth-decomp}
\begin{aligned}
(\d\th)^\top&=\Theta^\top-\tfrac12([\th^\top,\th^\top]+[\th^\perp,\th^\perp]),\\
(\d\th)^\perp&=\Theta^\perp-[\th^\top,\th^\perp].
\end{aligned}
\end{equation}
From this we compute
\[
\B{\theta}{\d\theta}=\B{\theta^{\top}+\theta^{\perp}}{(\d\theta)^{\top}+(\d\theta)^{\perp}}=\B{\theta^{\top}}{\d\theta^{\top}}-\B{\theta^{\perp}}{[\theta^{\top},\theta^{\perp}]}+\B{\theta^{\perp}}{\Theta^{\perp}}
\]
and
\[
\begin{aligned}
\B{\theta}{[\theta,\theta]}&=\B{\theta^{\top}+\theta^{\perp}}{[\theta^{\top}+\theta^{\perp},\theta^{\top}+\theta^{\perp}]}\\
&=\B{\theta^{\top}+\theta^{\perp}}{[\theta^{\top},\theta^{\top}]+[\theta^{\perp},\theta^{\perp}]+2[\theta^{\top},\theta^{\perp}]}\\
&=\B{\theta^{\top}}{[\theta^{\top},\theta^{\top}]}+\B{\theta^{\top}}{[\theta^{\perp},\theta^{\perp}]}+2\B{\theta^{\perp}}{[\theta^{\top},\theta^{\perp}]}.
\end{aligned}
\]
In total we thus have
\[
\begin{aligned}
\CS(\theta)&=\B{\theta}{\d\theta}+\tfrac{1}{3}\B{\theta}{[\theta,\theta]}\\
&=\CS(\theta^{\top})+\B{\theta^{\perp}}{\Theta^{\perp}}+\tfrac{1}{3}\B{\theta^{\top}}{[\theta^{\perp},\theta^{\perp}]}-\tfrac{1}{3}\B{\theta^{\perp}}{[\theta^{\top},\theta^{\perp}]}\\
&=\CS(\theta^{\top})+\B{\theta^{\perp}}{\Theta^{\perp}},
\end{aligned}
\]
where the last equality uses \eqref{eq:B-invar}. 
\end{proof}

\begin{remark}
In the applications of \cref{lem:CS-pairs} that we are interested in, we consider the case where the Lie algebra valued form $\theta$ is \emph{flat}, that is, satisfies the Maurer--Cartan equation $\Theta=\d\theta+\tfrac{1}{2}[\theta,\theta]=0$. In this case \eqref{eqn:partialblindness} simplifies to $\CS(\theta)=\CS(\theta^{\top})$, that is, the Chern--Simons $3$-form is blind to the $\theta^{\perp}$ component of $\theta$. 
\end{remark}

 As a first application of \cref{lem:CS-pairs} we assume that we have an inclusion
$i\col G\to\tilde G$ of groups corresponding to $\frak{g}\subset\tilde{\frak{g}}$ and get a
result on integrality of the Chern--Simons forms associated to the Maurer--Cartan
forms.
\begin{lemma}\label{lem:CS-subgroup}
Let $i\col G\to\tilde G$ be an inclusion of a Lie subgroup and consider the corresponding
subalgebra $\frak{g}\subset\tilde{\frak{g}}$. Assume that $\B{\,\cdot\,}{\cdot\,}$ is an
invariant bilinear form on $\tilde{\frak{g}}$ such that $\B{\,\cdot\,}{\cdot\,}|_{\frak{g}\x\frak
g}$ is non-degenerate.

If $\B{\,\cdot\,}{\cdot\,}$ is normalized in such a way that $[\CS(\mu_{\tilde G})]\in
H^3(\tilde G,\Z)$, then $[\CS^{\frak{g}}(\mu_G)]\in H^3(G,\Z)$.
\end{lemma}
\begin{proof}
Decomposing $\mu_{\tilde G}=\mu_{\tilde G}^\top+\mu_{\tilde G}^\perp$, the
definition of the Maurer--Cartan form readily implies that $i^*\mu_{\tilde
G}^\top=\mu_G$. Since $\mu_{\tilde G}$ satisfies the Maurer--Cartan equation,
\cref{lem:CS-pairs} gives that $\CS(\mu_{\tilde G})=\CS(\mu_{\tilde
G}^\top)$ and using \eqref{eqn:cspullback} we conclude that $i^*\CS(\mu_{\tilde
G}^\top)=\CS^{\tilde{\frak{g}}}(i^*\mu_{\tilde G}^\top)$. Since we use the
restriction of $\B{\,\cdot\,}{\cdot\,}$ on $\frak{g}$, this coincides with $\CS^{\frak
g}(\mu_G)$. Thus the result follows from the fact that pullbacks preserve
integral cohomology classes.
\end{proof}

There is a conceptual way to obtain compatible invariant bilinear forms as we need
them here. Assume that $\tilde{\frak{g}}$ is realized as a subalgebra of
$\frak{gl}(n,\R)$ for some $n$. Then it is well-known that the \textit{trace form}
$\B{X}{Y}\deq\tr(XY)$ defines a $\mathrm{GL}(n,\R)$-invariant bilinear form on
$\frak{gl}(n,\R)$. Of course, the restriction of $\B{\,\cdot\,}{\cdot\,}$ to $\tilde{\frak{g}}$
is the $\tilde G$-invariant for the subgroup $\tilde G\subset \mathrm{GL}(n,\R)$ corresponding
to the Lie subalgebra $\tilde{\frak{g}}$. If $\tilde{\frak g}$ is simple and the
restriction is non-zero, then it has to be a multiple of the Killing form and hence
is non-degenerate. Of course, we can restrict further to
$\frak{g}\subset\tilde{\frak{g}}$ and apply the same argument if $\frak{g}$ is also
simple.

Even better, for $m>n$, we can decompose $\R^m=\R^{m-n}\oplus\R^n$ and
then include $\mathrm{GL}(n,\R)\subset \mathrm{GL}(m,\R)$ as the maps that send $\R^n$ to itself and
are the identity on the complementary subspace $\R^{m-n}$. For the infinitesimal
inclusion $\frak{gl}(n,\R)\to\frak{gl}(m,\R)$ the trace form on $\frak{gl}(m,\R)$
evidently restricts to the trace form on $\frak{gl}(n,\Bbb R)$. This allows us to
also compare algebras of matrices of different sizes.

It is easy to write an explicit formula for Chern--Simons forms on subalgebras of
$\frak{gl}(n,\R)$ computed with respect to the trace form. One realizes a
$1$-form $\th$ on $N$ with values in such a subalgebra as a matrix
$(\th^i_j)_{i,j=1}^n$ of real-valued $1$-forms $\th^i_j\in\Om^1(N)$. Then of course,
we get $\d\th=(\d\th^i_j)_{i,j=1}^n$, and employing the summation convention, we can
write $[\th,\th]=(2\th^i_k\wedge\th^k_j)_{i,j=1}^n$. By definition of the Chern--Simons
form, this implies that we get (still using the summation convention)
\[
\CS(\th)=\th^i_j\wedge\d\th^j_i+\tfrac{2}{3}\th^i_j\wedge\th^j_k\wedge\th^k_i. 
\]
It is usual in Chern--Simons theory to write this as $\tr(\th\wedge
\d\th+\tfrac{2}{3}\th\wedge\th\wedge\th)$. Using this, we can prove a result on the
Lie algebras we will need in the further developments: 

\begin{prop}\label{prop:sl-so}
For $n\geqslant 3$, consider the subgroups $G_n\deq\mathrm{SO}(n,\R)\subset\tilde G_n\deq\mathrm{SL}(n,\Bbb
R)\subset \mathrm{GL}(n,\R)$ and denote the Maurer--Cartan forms by $\mu_n=\mu_{G_n}$
and $\tilde\mu_n=\mu_{\tilde G_n}$. Then for $C_n\in\R$ the statements 
\begin{gather}
C_n\tr\left(\mu_n\wedge \d\mu_n+\frac{2}{3}\mu_n\wedge\mu_n\wedge\mu_n\right)\in
H^3(G_n,\Z)\label{mu_so}\\ C_n\tr\left(\tilde\mu_n\wedge
\d\tilde\mu_n+\frac{2}{3}\tilde\mu_n\wedge\tilde\mu_n\wedge\tilde\mu_n\right)\in
H^3(\tilde G_n,\Z)\label{mu_sl}.
\end{gather}
are equivalent and for any $N\geqslant 5$, there exists a constant $C$ such that they hold
with $C_n=C$ for all $n\leqslant N$.
\end{prop}
\begin{proof}
It is well-known that for both $\frak{g}_n=\frak{so}(n)$ and $\tilde{\frak
g}=\frak{sl}(n,\R)$ the trace-form is a non-zero multiple of the Killing
form, see e.g.\ \cite{MR1153249}, so it is always non-degenerate. It is also well
known that for $n\geqslant 2$, $G_n\subset\tilde G_n$ is a strong deformation retract by
Iwasawa decomposition, so the inclusion induces an isomorphism in cohomology with
both real and integer coefficients. Finally, for $n\neq 4$, $H^3(\tilde G_n,\Bbb
R)$ is $1$-dimensional, so we can choose $C_n$ in such a way that \eqref{mu_sl}
holds for $n$. But then the full result immediately follows from
\cref{lem:CS-subgroup}.
\end{proof}
 The explicit normalizations we will use are computed in \cref{Appendix}
below.

\section{Flat extensions of principal
connections}\label{sect:flatextensionGequlasH}

 We continue to work in the setting of an inclusion $i\col G\to\tilde G$ of
a Lie subgroup and an invariant bilinear form $\B{\,\cdot\,}{\cdot\,}$ on $\tilde{\frak{g}}$
whose restriction to $\frak{g}$ is non-degenerate. All Chern--Simons forms will be
formed with respect to this bilinear form from now on. In this setting recall from \cref{defn:flatext} our key concept of a flat extension:

\begin{defNON}\label{def:flat-ext}
Let $\pi : P \to M$ be a principal $G$-bundle and $\theta \in \Omega^1(P,\mathfrak{g})$ a connection. A \emph{flat extension of $\theta$ of type $(\tilde{G},G)$} is a bundle homomorphism $F : P \to \tilde{G}$ into the total space of the principal $G$-bundle $\tilde{G} \to \tilde{G}/G$ so that $\theta=F^*(\mu_{\tilde{G}}^{\top})$.
\end{defNON}

\begin{remark}\label{rem:flat-ext}
The concept of flat extensions as defined here makes sense also in the case of a $\frak g$-connection on an $H$-principal bundle with $H\subset G$ as in \cref{1.3}. We restrict to the case $G=H$ here, however, since for $H\neq G$, integrality of the cohomology class $[\CS(\mu_{\tilde G})]$ does not imply integrality of $[\CS^{\frak g}(\mu_H)]$, so additional assumptions are needed to obtain a general result on integrality or vanishing of Chern--Simons invariants in that setting. \end{remark}

\begin{thm}\label{thm:main-flat}
Consider a pair $G\subset\widetilde{G}$ and a bilinear form $\B{\,\cdot\,}{\cdot\,}$ on
$\tfg$ which is normalized in such a way that $[\CS(\mu_{\tilde G})]\in
H^3(\tilde G,\Z)$. Let $\pi\col P\to M$ be a principal $G$-bundle and
$\theta\in\Om^1(P,\frak{g})$ a $G$-equivariant form.

Suppose that $\th$ admits a flat extension of type $(\widetilde{G},G)$
such that for the corresponding map $F:P\to\tilde G$, we get 
\begin{equation}\label{cond:flatextcond}
[F^*(\mu_{\tilde G}^\perp),F^*(\mu_{\tilde G}^\perp)]\in\Om^2(P,\frak{g}).
\end{equation}
Then for any global smooth section $\si$ of $P$, we have $\int_M\si^*\CS(\theta)\in \Bbb
Z$. If $\CS(\mu_{\widetilde{G}})$ is exact, then we even have
$\int_M\si^*\CS(\theta)=0$.
\end{thm}
 
\begin{remark}
Since the condition \eqref{cond:flatextcond} is automatically satisfied if $(\tilde{\mathfrak{g}},\mathfrak{g})$ is a symmetric pair, \cref{thm:main-flat} implies \cref{cor:keystatement}.
\end{remark}

\begin{proof}
 Observe that by \cref{lem:CS-subgroup}, our assumptions imply that $[\CS(\mu_G)]\in
 H^3(G,\Z)$, so our construction of invariants from \cref{1.3} can be applied. For a
 flat extension $F\col P\to \tilde G$, consider $F^*\mu_{\tilde G}\in
 \Om^1(P,\tfg)$ which by definition satisfies
 $\theta=F^*(\mu_{\tilde{G}}^\top)=(F^*\mu_{\tilde G})^\top$.  Since
   $[F^*(\mu_{\tilde G}^\perp),F^*(\mu_{\tilde G}^\perp)]\in\Om^2(P,\frak{g})$, we
   can apply \cref{lem:CS-pairs} to $F^*\mu_{\tilde G}$ to conclude that
   $\CS(\th)=\CS(F^*\mu_{\tilde G})$ which in turn equals $F^*(\CS(\mu_{\tilde G}))$
   by \eqref{eqn:cspullback}. Hence
   $\si^*(\CS(\theta))=(F\o\si)^*(\CS(\mu_{\widetilde{G}}))$ and, as a pullback of an
   integral class, this represents a class in $H^3(M,\Z)$, so the result follows.
\end{proof}

Thus we see that in the cases where we obtain $\R/\Z$-valued invariants, existence of a flat extension implies vanishing of the invariant. In the case of $\R$-valued invariants, the consequence of flat extensions depends on the topology of $\widetilde{G}$. In case that $H^3(\widetilde{G},\R)=\{0\}$ (or more generally $[\mu_{\widetilde{G}}]=0$), we again get vanishing of the invariant. However, if $0\neq [\mu_{\widetilde{G}}]\in H^3(\widetilde{G},\Z)$, then we can only conclude that the invariant is an integer.

\subsection{Geometric construction of Riemannian flat extensions}\label{1.6}
The basic example of a flat extension comes from the original article of
Chern--Simons. Suppose that $(M,g)$ is an oriented Riemannian $3$-manifold, so $G=H=\mathrm{SO}(3)$. Then we choose $\tilde G=\mathrm{SO}(4)$, which contains $G$ as
a subgroup in the obvious way, i.e.\ as matrices of the form 
\[\begin{pmatrix} 1 & 0
  \\ 0 & A \end{pmatrix}.
  \]
One immediately verifies that for the Lie algebras this
implies that $(\tfg,\frak{g})$ is a symmetric pair.

Now $\tilde G$ acts transitively on the space of oriented $3$-planes in the inner
product space $\R^4$ and for this action, $G$ is the stabilizer of the oriented
subspace $\R^3\subset\R^4$. Hence we can identify the symmetric space
$\mathrm{SO}(4)/\mathrm{SO}(3)$ as the space of all oriented $3$-planes in
$\R^4$.  An oriented $3$-plane in $\R^4$ is of course uniquely
  determined by its positive unit normal, which leads to the more common
  identification $\mathrm{SO}(4)/\mathrm{SO}(3)\cong S^3$. 

 As the $\mathrm{SO}(4)$-invariant inner product on $\tfg$, we use
  $\tfrac{1}{16\pi^2}$ times the trace form. \cref{prop:normalization} in
  \cref{Appendix} below shows that this is a minimal choice for which we get
  $[\CS(\mu_{\mathrm{SO}(4)})]\in H^3(\mathrm{SO}(4),\Z)$. By \cref{lem:CS-subgroup},  
  the restriction of this form to $\mathfrak{so}(3)$ has the property that
  $[\CS(\mu_{\mathrm{SO}(3)})]\in H^3(\mathrm{SO}(3),\Z)$. More precisely, 
  \cref{prop:normalization} actually shows that
  $\int_{\mathrm{SO}(3)}\CS(\mu_{\mathrm{SO}(3)})=1$, so this is an optimal choice of
  normalization for obtaining an invariant with values in $\R/\Z$.

\begin{prop}\label{prop:Riemann-flat-ext}
Suppose that $(M,g)$ is a closed oriented Riemannian $3$-manifold and let
$\theta\in\Om^1(\Cal{SO}M,\frak{g})$ be the Levi-Civita connection form of $g$ on the
oriented orthonormal frame bundle of $M$. Then an isometric immersion $i\col M\to
\R^4$ gives rise to a flat extension of $\theta$ of type
$(\mathrm{SO}(4),\mathrm{SO}(3))$ and hence the Chern--Simons invariant of $(M,g)$
vanishes.
\end{prop}
\begin{proof}
Let $f$ be the Gauss map of $i$, which sends a point $x\in M$ to the $3$-plane
$T_xi(T_xM)$ with the orientation induced from the one of $T_xM$  (or
  alternatively to the positive unit normal), viewed as a smooth map $M\to
\mathrm{SO}(4)/\mathrm{SO}(3)$. Since we start from an isometric immersion the map
$T_xi$ is orthogonal for $g_x$ and the restriction of the standard inner product to
$T_xi(T_xM)$. A point in $\Cal{SO}M$ over $x\in M$ can be viewed as an orientation
preserving linear isomorphism $u\col \R^3\to T_xM$ which is orthogonal with respect
to the standard inner product on $\R^3$ and $g_x$. Hence $T_xi\o u$ is an orthogonal
isomorphism from $\R^3$ to $T_xi(T_xM)$ and using the positively oriented unit normal
we can extend this to an orientation preserving orthogonal isomorphism from $\R^4$ to
itself, i.e.\ to an element of $\mathrm{SO}(4)$. This defines a smooth map $F\col
\Cal{SO}M\to \mathrm{SO}(4)$ which lifts the Gauss map $f$ and by construction is
$\mathrm{SO}(3)$-equivariant and hence a homomorphism of principal bundles. Hence
$F^*\mu_{\mathrm{SO}(4)}^\top\in\Om^1(\Cal{SO}M,\frak{g})$ is a principal connection
form on $\Cal{SO}M$, so we only have to prove that it coincides with the Levi-Civita
connection.

The orthonormal frame bundle of $\R^4$ is the trivial bundle $\R^4\x \mathrm{SO}(4)$
and viewed as a $1$-form on this bundle $\mu_{\mathrm{SO}(4)}$ is the Levi-Civita
connection form for the flat metric on $\R^4$. Now consider a vector field
$\tilde\xi$ on $\R^4$ and the corresponding $\mathrm{SO}(4)$-equivariant function
$\Bbb R^4\x \mathrm{SO}(4)\to\R^4$. Splitting this into a $\R^3$-component and an
$\Bbb R$-component and composing with $F$, we obtain maps $\Cal{SO}M\to\R^3$ and
$\Cal{SO}M\to\R$, which are $\mathrm{SO}(3)$-equivariant and
$\mathrm{SO}(3)$-invariant, respectively. So the first component defines a vector
field $\xi$ on $M$ while the second descends to a smooth function $a\col
M\to\R$. By construction, these have the property that
$\tilde\xi(i(x))=T_xi(\xi(x))+a(x)\frak{n}(x)$, where $\frak{n}(x)$ is the positively
oriented unit normal to $T_xi(T_xM)$. But this implies that the covariant derivative
induced by $F^*\mu_{\mathrm{SO}(4)}^\top$ coincides with the tangential part of the
covariant derivative in $\R^4$ and since $i$ is isometric, this coincides with the
Levi-Civita connection of $M$.
\end{proof}

\begin{remark}\label{rem:Rieman}
Consider a flat extension $F\col\Cal{SO}M\to \mathrm{SO}(4)$ as in
\cref{prop:Riemann-flat-ext}. Since $F$ maps the fundamental vector field generated
by $X\in\frak{g}=\mathfrak{so}(3)$ to left-invariant vector field $L_X$ on
$\mathrm{SO}(4)$, it follows that $F^*(\mu_{\mathrm{SO}(4)}^\perp)$ is not only
$\mathrm{SO}(3)$-equivariant but also horizontal. Since
$\frak{g}^\perp\cong\R^3\cong\R^{3*}$ as a representation of $\mathrm{SO}(3)$, we can
identify $F^*(\mu_{\mathrm{SO}(4)}^\perp)$ as a $1$-form on $M$ with values in
$TM\cong T^*M$. From the description in the proof of \cref{prop:Riemann-flat-ext},
one readily concludes that in the first interpretation, this produces the classical
\textit{shape operator}, which maps $\xi\in\frak{X}(M)$ to the tangential component
of the derivative of the unit normal in direction $\xi$. In the second
interpretation, one obtains the \textit{second fundamental form}, which maps
$\xi,\eta\in\frak{X}(M)$ to the normal component of the derivative of $\eta$ in
direction $\xi$. The fact that the shape operator and the second fundamental form are
essentially the same object in this case is a consequence of the unit normal being
orthogonal to the tangent spaces.
\end{remark}

\subsection{Lorentzian flat extensions}\label{1.7}
The line of argument of \cref{1.6} carries over directly to the Lorentzian case, but
things become more interesting here. On the one hand, so far no normalization of the
invariant bilinear form was needed in this case, but this becomes important now. On
the other hand, we have now two possibilities for isometric immersions, namely either
into $\mathbb R^{3,1}$ (with positive normal) or into $\mathbb R^{2,2}$ (with
negative normal). As we shall see, these have different consequences for the
invariant. The corresponding choices for $\tilde G$ are $\mathrm{SO}_0(3,1)$ and
$\mathrm{SO}_0(2,2)$ respectively, with the obvious inclusion of
$G=\mathrm{SO}_0(2,1)$ in both cases. In both cases, one obtains a symmetric pair
$(\tfg,\frak{g})$. The homogeneous space $\mathrm{SO}_0(3,1)/\mathrm{SO}_0(2,1)$ can
be interpreted as the space of all space- and time-oriented Lorentzian linear
subspaces $V\subset\R^{3,1}$. For any such subspace there is a unique oriented
(positive) unit normal. Similarly, $\mathrm{SO}_0(2,2)/\mathrm{SO}_0(2,1)$ is the
space of all space- and time-oriented Lorentzian linear subspaces $V\subset\R^{2,2}$
and such a space determines a unique (negative) oriented unit normal.

The difference between the two cases comes from the structure of $H^3(\tilde G,\Bbb
R)$. For $\mathrm{SO}_0(3,1)$, the maximal compact subgroup is isomorphic to
$\mathrm{SO}(3)$, so the third cohomology is non-trivial here. For
$\mathrm{SO}_0(2,2)$, the maximal compact subgroup is two dimensional and hence the
third cohomology vanishes identically. Thus the appropriate normalization for our
purposes is to start with an invariant non-degenerate symmetric bilinear form on
$\mathfrak{so}(3,1)$ normalized in such a way that
$[\CS(\mu_{\mathrm{SO}_0(3,1)})]\in H^3(\mathrm{SO}_0(3,1),\Z)$ for the resulting
normalization.  As we have seen in \cref{1.6} above,
  $\tfrac{1}{16\pi^2}$ times the trace form leads to an optimal normalization for
  $\mathrm{SO}(3)$. By the Iwasawa decomposition, $\mathrm{SO}(3)$ is a deformation retract
  of $\mathrm{SO}_0(3,1)$, so we can use the same normalization here. Now the proof
of \cref{prop:Riemann-flat-ext} generalizes in a straightforward way to show

\begin{prop}\label{prop:Lorentz-flat-ext}
Suppose that $(M,g)$ is a closed space- and time-oriented Lorentzian $3$-ma\-ni\-fold which admits a global orthonormal frame and let $\theta\in\Om^1(\Cal{SO}_0M,\frak{g})$ be the Levi-Civita connection form of $g$ on the space- and time-oriented orthonormal frame bundle of $M$.

(1) An isometric immersion $i\col M\to \R^{2,2}$ gives rise to a flat extension of
$\theta$ of type $(\mathrm{SO}_0(2,2),\mathrm{SO}_0(2,1))$ and hence the ($\R$-valued) Chern--Simons invariant of
$(M,g)$ vanishes.

(2) An isometric immersion $i\col M\to \R^{3,1}$ gives rise to a flat extension of
$\theta$ of type $(\mathrm{SO}_0(3,1),\mathrm{SO}_0(2,1))$ and hence the ($\R$-valued) Chern--Simons
invariant of $(M,g)$ has to be an integer.
\end{prop}

The interpretation of $F^*(\mu_{\tilde G}^\perp)$ for a flat extension $F$ is
completely parallel to the Riemannian case discussed in \cref{rem:Rieman}. 

\subsection{Equiaffine immersions and flat extensions}\label{1.8}
We next discuss the case of volume preserving affine connections on $3$-manifolds,
so $G=H=\mathrm{SL}(3,\R)$. Similarly to the pseudo-Riemannian case, we consider
$\widetilde{G}\deq\mathrm{SL}(4,\R)$ and the obvious inclusion of
$\mathrm{SL}(3,\R)$ here and the flat extension will be obtained from a Gauss
map. The details will be quite a bit more involved, though, in particular $\tilde
G/G$ has dimension $7$ and is not a symmetric space. Indeed, on the level of Lie
algebras, the inclusion $\frak g\deq\frak{sl}(3,\R)\hookrightarrow
\frak{sl}(4,\R)=:\tfg$ corresponds to a block decomposition of elements of $\tfg$ of
the form
\[
 \begin{pmatrix} a & Z \\ X & A-\tfrac{a}{3}\Bbb
    I\end{pmatrix}
\] 
with $A\in\frak{g}$, $a\in\R$, $X\in\R^3$ and
$Z\in\R^{3*}$. Hence we conclude that $\frak{g}^\perp=\Bbb R\oplus\R^3\oplus\R^{3*}$
as a representation of $\frak{g}$, and the component
$\frak{g}^\perp\x\frak{g}^\perp\to\frak{g}^\perp$ of the Lie bracket of $\tfg$ is (in
obvious notation) given by
\begin{multline}\label{eq:non-symm}
[(a_1,X_1,Z_1),(a_2,X_2,Z_2)]\\
=(Z_1X_2-Z_2X_1,\frac43(-a_1X_2+a_2X_1),\frac43(a_1Z_2-a_2Z_1)). 
\end{multline}

Let us first give an appropriate interpretation of the homogeneous space $\tilde
G/G$. Consider the set of all pairs $(V,\ell)$, where $V\subset\R^4$ is an
oriented linear subspace of dimension $3$ endowed with a volume element and
$\ell\subset\R^4$ is a $1$-dimensional linear subspace complementary to $V$.
Then $\mathrm{SL}(4,\R)$ acts on this space by $A\cdot (V,\ell)\deq(A(V),A(\ell))$. Starting
from the natural base point given by $V=\R^3$ with the orientation and volume
element induced by the standard basis $\{e_1,e_2,e_3\}$ and $\ell=\R\cdot e_4$,
linear algebra shows that the action of $\mathrm{SL}(4,\R)$ is transitive and that the
stabilizer of the base point is $\mathrm{SL}(3,\R)$. Hence we may identify $\tilde G/G$
with the space of all such pairs.

Now let us assume that we have given an oriented $3$-manifold $M$. Suppose that
$i\col M\to\R^4$ is an immersion and that in addition we choose a transversal,
i.e.~for each $x\in M$, we choose a line $\ell(x)\subset\Bbb R^4$, which is
complementary to the $3$-dimensional subspace $T_xi(T_xM)\subset\Bbb
R^4=T_{i(x)}\R^4$. We assume this choice to be smooth in the obvious sense,
i.e.~locally around each point the lines can be spanned by a smooth vector field
along $i$. This shows that $i$ and $\ell$ determine a smooth Gauss map $f\col
M\to\tilde G/G$, which sends each $x\in M$ to the pair $(T_xi(T_xM),\ell(x))$.

Via $i$ and $\ell$, any tangent vector at a point $i(x)$ decomposes uniquely into a
component in $\ell(x)$ and a component in $T_xi(T_xM)$. Via this decomposition, the
flat connection on $\R^4$ induces a linear connection $\nabla^\top$ on $TM$ as
well as a connection $\nabla^\perp$ on the trivial line bundle $M\x\R$. It is
easy to understand when $\nabla^\top$ preserves a volume form.

\begin{lemma}\label{lemma:volume}
Given an immersion $i\col M\to\R^4$ and a choice $\ell$ of transversal, the induced
connection $\nabla^\top$ preserves a volume form if and only if there is a global
non-zero section of $M\x\R$ which is parallel for $\nabla^\perp$.
\end{lemma}
\begin{proof}
The standard volume form on $\R^4$ defines an isomorphism $\Lambda^4\Bbb
R^4\to\R$. The decomposition $\R^4=T_xi(T_xM)\oplus\ell(x)$ induces an
isomorphism $\Lambda^4\R^4\cong \Lambda^3(T_xi(T_xM))\otimes\ell(x)$ for each
$x\in M$. Over $M$, this gives rise to a trivialization of the bundle $\La^3TM\otimes
L$, where $L=M\x\R$ and hence to an isomorphism $\La^3T^*M\cong L$. Since the
trivialization of $\La^4\R^4$ is compatible with the flat connection on $\Bbb
R^4$ this isomorphism pulls back the connection $\nabla^\perp$ on $L$ to the
connection on $\La^3T^*M$ induced by $\nabla^\top$. By definition, $\nabla^\top$
preserves a volume form if and only if the latter connection admits a non-zero global
parallel section.
\end{proof}

Now we introduce the classical concept of an equiaffine immersion, as in
\cite[Defintion 1.4]{MR1311248}.
\begin{definition}\label{def:equiaffine}
Let $M$ be a smooth manifold of dimension $3$ endowed with a fixed
  volume form $\nu\in\Om^3(M)$ and a torsion-free linear connection $\nabla$ on $TM$
which preserves $\nu$. An \textit{equiaffine} immersion of $M$ into
$\R^4$ is given by a smooth immersion $i\col M\to\R^4$ and a choice $\ell\col
M\to\R P^3$ of transversal such that the induced connection
$\nabla^\top$ on $TM$ coincides with $\nabla$.
\end{definition}

 As the invariant bilinear form on $\tfg=\mathfrak{sl}(4,\R)$, we again
  use $\tfrac{1}{16\pi^2}$ times the trace form. In view of \cref{prop:sl-so} the
  considerations in \cref{1.6} show that this leads to
  $[\CS(\mu_{\mathrm{SL}(4,\R)})]\in H^3(\mathrm{SL}(4,\R),\Z)$ as well as
  $\int_{\mathrm{SO}(3)}\CS(\mu_{\mathrm{SL}(3,\R)})=1$.

\begin{thm}\label{thm:equiaffine}
Let $M$ be a closed oriented $3$-manifold endowed with a  volume form
  $\nu\in\Om^3(M)$ and a torsion-free linear connection $\nabla$ on $TM$ which
preserves $\nu$. Let $\Cal{SL}M$ be the volume preserving frame bundle of $M$ with
respect to $\nu$ and let $\theta\in\Om^1(\Cal{SL}M,\mathfrak{sl}(3,\R))$ be the
connection form of $\nabla$.

Then an equiaffine immersion of $M$ into $\R^4$ defines a flat extension of $\theta$
of type $(\mathrm{SL}(4,\R),\mathrm{SL}(3,\R))$ that satisfies
\eqref{cond:flatextcond} and hence implies vanishing of the Chern--Simons invariant
associated to $\nabla$.
\end{thm}
\begin{proof}
We follow the same route as in the Riemannian case with appropriate modifications. A
point $u\in \Cal{SL}M$ over $x\in M$ is a linear isomorphism $\R^3\to T_xM$ which
maps the standard basis $\{e_2,e_3,e_4\}$ to a positively oriented basis of unit
volume. Composing this with $T_xi\col T_xM\to\R^4$, there exists a unique vector
$0\neq \frak{t}(x)\in\ell(x)\subset\R^4$ which completes
$T_xi(u(e_2)),T_xi(u(e_3)),T_xi(u(e_4))$ to a positively oriented basis of $\R^4$ of
unit volume. Choosing $F(u)(e_1)=\frak{t}(x)$ extends $u$ to a linear isomorphism
$F(u)\col\R^4\to\R^4$ which is orientation-preserving and volume-preserving and hence
an element of $\tilde G=\mathrm{SL}(4,\R)$. Clearly, this defines a smooth map
$F\col\Cal{SL}M\to \tilde G$ which is $G$-equivariant, where
$G=\mathrm{SL}(3,\R)$. Thus it defines a morphism of principal fibre bundles to
$\tilde G\to \tilde G/G$.

As we have noted above, the flat connection on $\R^4$ induces linear connections
$\nabla^\top$ on $TM$ and $\nabla^\perp$ on $M\x\R$. Since we started with an
equiaffine immersion, $\nabla^\top=\nabla$ and since this is volume preserving, 
$\nabla^\perp$ is flat by \cref{lemma:volume}. Now as in the proof of
\cref{prop:Riemann-flat-ext}, the first property implies that
$\theta=F^*(\mu_{\widetilde G}^\top)\in\Om^1(\Cal{SL}M,\frak{g})$. The second property
implies that the $1$-form $F^*(\mu_{\widetilde G}^\perp)$ has vanishing $\Bbb
R$-component. From formula \eqref{eq:non-symm} we conclude that this implies that the
$\frak{g}^\perp$-component of $[F^*(\mu_{\widetilde G}^\perp),F^*(\mu_{\widetilde
G}^\perp)]$ has values in the $\R$-component of $\frak{g}^\perp$ only. But on
the other hand, we can take the $\frak{g}^\perp$-component of the Maurer--Cartan
equation for $\mu_{\widetilde G}$ and pull back by $F$ to conclude that the latter
component can be written as
\[
-2\d F^*(\mu_{\widetilde G}^\perp)-2[F^*(\mu_{\widetilde G}^\top),F^*(\mu_{\widetilde
G}^\perp)], 
\]
which by construction has vanishing $\R$-component. Thus we have found the
claimed flat extension and the vanishing of the Chern--Simons invariant follows from \cref{thm:main-flat} and the comparison of normalization conditions on
$\B{\,\cdot\,}{\cdot\,}$ above.
\end{proof}

\begin{remark}\label{rem:equi-aff}
(1) As before, the two non-vanishing components of $F^*(\mu_{\widetilde G}^\perp)$ are horizontal, $G$-equivariant $1$-forms with values in $\R^3$ respectively in $\Bbb R^{3*}$. From the description of the induced connection in the proof of \cref{thm:equiaffine} we see that these are again the shape operator in $\Om^1(M,TM)$ and the second fundamental form in $\Om^1(M,T^*M)$ of the equiaffine immersion. In contrast to the (pseudo-)Riemannian case, these two objects are independent now, their only relation is that wedge-contraction of the two forms (which is an element of $\Om^2(M)$) vanishes, which follows immediately from the volume compatibility of the immersion.

(2) If $(M,g)$ is an oriented Riemannian $3$-manifold and $i:M\to\Bbb
  R^4$ is an isometric immersion, then using the normal as a transversal makes $i$
  into an equiaffine immersion. Since both results use the same multiple of the
  trace form, this shows that \cref{prop:Riemann-flat-ext} actually is a special case of
  \cref{thm:equiaffine}.
\end{remark}

\appendix

\section{Normalization}\label{Appendix}

 We can apply the construction of the Riemannian Gauss map from
  \cref{1.6} (in general dimensions) to the inclusion $i:S^n\to\Bbb R^{n+1}$ of the
  standard sphere. Lifting this to the oriented orthonormal frame bundle
  $\Cal{SO}S^n$ one obtains an isomorphism $F:\Cal{SO}S^n\to \mathrm{SO}(n+1)$
  covering the diffeomorphism $S^n\to\mathrm{SO}(n+1)/\mathrm{SO}(n)$ discussed in
  \cref{1.6}. This is most easily interpreted as taking the map
  $p:\mathrm{SO}(n+1)\to S^n$ defined by mapping $A\in\mathrm{SO}(n+1)$ to $A(e_1)$,
  i.e.~the first column vector of $A$. This is the projection of the oriented
  orthonormal frame bundle, since the remaining columns of $A$ are a positively
  oriented orthonormal basis for $A(e_1)^\perp=T_{A(e_1)}S^n$. The considerations in
  \cref{1.6} also show that for the Maurer-Cartan form $\mu_{n+1}$ of
  $\mathrm{SO}(n+1)$, $\mu_{n+1}^\top$ is the connection form of the Levi-Civita
  connection of the round metric on $S^n$, while $\mu_{n+1}^\perp$ represents the
  second fundamental form.

\begin{prop}\label{prop:normalization}
Let $\mu_4$ denote the Maurer--Cartan form of $\mathrm{SO}(4)$. Then
\begin{equation}\label{eqn:goodnorm}
\zeta\deq\frac{1}{16\pi^2}\tr\left(\mu_4\wedge \d
\mu_4+\tfrac{2}{3}\mu_4\wedge\mu_4\wedge\mu_4\right)
\end{equation}
represents an element of $H^3(\mathrm{SO}(4),\Z)$.  Moreover, for the
  inclusion $\iota:\mathrm{SO}(3)\to \mathrm{SO}(4)$, we have
  $\int_{\mathrm{SO}(3)}\iota^*\zeta=1$.
\end{prop}
\begin{proof}
We will show that the $3$-form defined in \eqref{eqn:goodnorm} evaluated on a set of
generators of $H_3(\mathrm{SO}(4),\Z)$ is $1$.  The principal right
$\mathrm{SO}(3)$-bundle $\pi \col \mathrm{SO}(4) \to S^3$ is trivial so that
$\mathrm{SO}(4)$ is diffeomorphic to $S^3\times \mathrm{SO}(3)$. Standard results
about the homology of $S^3$ and $\mathrm{SO}(3)\simeq \mathbb{RP}^3$ together with
K\"unneth's theorem imply that
\[
H_3(\mathrm{SO}(4),\Z)\simeq H_3(S^3\times \mathrm{SO}(3),\Z)\simeq \Z \oplus \Z. 
\]
Moreover, $H_3(\mathrm{SO}(4),\Z)$ is generated by the pushforward of the fundamental
class of $\mathrm{SO}(3)$ under the inclusion $\iota \col \mathrm{SO}(3) \to
\mathrm{SO}(4)$ and the pushforward of the fundamental class of $S^3$ under a section
$\sigma \col S^3 \to \mathrm{SO}(4)$. It is thus sufficient to show that we have
\begin{equation}\label{eqn:integrals}
\int_{\mathrm{SO}(3)}\iota^*\zeta=1 \qquad \text{and}\qquad \int_{S^3}\sigma^*\zeta=1
\end{equation}
with respect to suitable orientations of $S^3$ and $\mathrm{SO}(3)$. We first treat
the left integral.  We know from the proof of \cref{lem:CS-subgroup}
  that $\iota^*\mu_4^\top=\mu_3$, the Maurer--Cartan form of
$\mathrm{SO}(3)$. Furthermore, $(\tilde{\mathfrak{g}},\mathfrak{g})$ is a symmetric
pair so that \cref{lem:CS-pairs} implies
\[
\int_{\mathrm{SO}(3)}\iota^*\zeta=\frac{1}{16\pi^2}\int_{\mathrm{SO}(3)}\tr\left(\mu_3\wedge\d\mu_3+\tfrac{2}{3}\mu_3\wedge\mu_3\wedge\mu_3\right).
\]
Writing
\[
\mu_3=\begin{pmatrix} 0 & -\omega_1 & -\omega_2 \\ \omega_1 & 0 & -\psi \\ \omega_2 & \psi & 0 \end{pmatrix}
\]
for left-invariant $1$-forms $\omega_1,\omega_2,\psi \in \Omega^1(\mathrm{SO}(3))$, an elementary computation gives
\[
\tr\left(\mu_3\wedge\d\mu_3+\tfrac{2}{3}\mu_3\wedge\mu_3\wedge\mu_3\right)=2\,\omega_1\wedge\omega_2\wedge\psi.
\]
Under the identification $\mathrm{SO}(3)\cong\Cal{SO}S^2$ observed above, the integral 
\[
\int_{\mathrm{SO}(3)}\omega_1\wedge\omega_2\wedge\psi
\] 
equals the product of the length of the typical fibre of $\mathrm{SO}(2)\to \mathcal{SO}S^2 \to S^2$ with the surface area of $S^2$. The former is given by $2\pi$ and the latter by $4\pi$. In summary we thus have
\[
\int_{\mathrm{SO}(3)}\iota^*\zeta=\frac{2}{16\pi^2}\int_{\mathrm{SO}(3)}\omega_1\wedge\omega_2\wedge\psi=\frac{1}{8\pi^2}\int_{\mathrm{SO}(3)}\omega_1\wedge\omega_2\wedge\psi=1. 
\]

In order to compute the right integral of \eqref{eqn:integrals} we consider the smooth $\pi$-section $\sigma \col S^3 \to \mathrm{SO}(4)$ given by the rule
\[
x=\begin{pmatrix} x_1 \\ x_2 \\ x_3 \\ x_4 \end{pmatrix} \mapsto \begin{pmatrix} x_1 & -x_2 & -x_3 & -x_4 \\ x_2 & x_1 & x_4 & -x_3 \\ x_3 & -x_4 & x_1 & x_2 \\ x_4 & x_3 & -x_2 & x_1 \end{pmatrix}
\]
for all $x \in S^3$. A tedious but straightforward calculation shows that
\[
\sigma^*\mu_4=\begin{pmatrix} 0 & \kappa & -\xi & -\rho  \\ -\kappa & 0 & \rho & -\xi \\ \xi & -\rho & 0 & -\kappa \\ \rho & \xi & \kappa & 0 \end{pmatrix}
\]
where $\xi,\rho,\kappa \in \Omega^1(S^3)$ are given by
\begin{equation}
\begin{aligned}\label{eqn:leftinvariantforms}
\xi&=-x_3\d x_1+x_4\d x_2+x_1\d x_3-x_2 \d x_4,\\
\rho&=-x_4\d x_1-x_3\d x_2+x_2\d x_3+x_1\d x_4,\\
\kappa&=x_2\d x_1-x_1\d x_2+x_4\d x_3-x_3 \d x_4.
\end{aligned}
\end{equation}
From this one computes
\[
\sigma^*\zeta=\frac{8}{16\pi^2}\,\xi\wedge\rho\wedge\kappa=\frac{1}{2\pi^2}\,\xi\wedge\rho\wedge\kappa
\]
and
\begin{multline}\label{eqn:volformthreesphere}
\xi\wedge\rho\wedge \kappa=x_4\d x_1\wedge \d x_2 \wedge \d x_3-x_3\d x_1\wedge\d x_2\wedge \d x_4\\
+x_2\d x_1\wedge \d x_3\wedge \d x_4-x_1\d x_2\wedge \d x_3 \wedge\d x_4,
\end{multline}
where we use that $|x|^2=1$. The right hand side of \eqref{eqn:volformthreesphere} equals the restriction to $S^3$ of the interior product of the inwards pointing radial vector field $-\sum_{i}x_i\partial_i$ and the standard volume form $\d x^1\wedge\d x^2\wedge \d x^3\wedge \d x^4$ of $\R^4$. This is the standard volume form of $S^3$ with respect to the inwards orientation. We thus have
\begin{equation}\label{eqn:volumethresphere}
\int_{S^3}\xi\wedge \rho\wedge \kappa=2\pi^2
\end{equation}
and the claim follows.
\end{proof}

\section{Examples}\label{app:examples}

\subsection{Lorentzian metrics}

We compute the Chern--Simons invariant for a $1$-parameter family of left-invariant Lorentzian metrics on $\mathrm{SU}(2)$. Under the identification $\mathrm{SU}(2)\simeq S^3\subset \R^4$ given by
\[
\begin{pmatrix} x_1+\i x_2 & -x_3+\i x_4 \\ x_3+\i x_4 & x_1-\i x_2 \end{pmatrix} \mapsto \begin{pmatrix} x_1 \\ x_2 \\ x_3 \\ x_4 \end{pmatrix}=x
\]
for all $x \in \R^4$ with $|x|=1$, the Maurer--Cartan form of $\mathrm{SU}(2)$ can be written as
\[
\mu:=\mu_{\mathrm{SU}(2)}=\begin{pmatrix} -\i \kappa & -\xi+\i \rho \\ \xi+\i\rho & \i \kappa \end{pmatrix},
\]
where $\xi,\rho,\kappa$ are given by \eqref{eqn:leftinvariantforms}. The Maurer--Cartan equation $\d\mu+\tfrac{1}{2}[\mu,\mu]=0$ is equivalent to the structure equations
\[
\d\xi=-2\rho\wedge\kappa, \qquad \d\rho=-2\kappa\wedge\xi, \qquad \d\kappa=-2\xi\wedge\rho.
\]
For $\lambda \in \R\setminus\{0\}$ consider the Lorentzian metric
$g_{\lambda}=\xi^2+\rho^2-\lambda^2\kappa^2$. This can be viewed as a natural
Lorentzian analog of a Berger sphere or of a pseudo-Hermitian structure on the
CR-manifold $\mathrm{SU}(2)$ induced by a left-invariant contact form. The Levi-Civita
connection form $\theta_{\lambda}$ of $g_{\lambda}$ -- when pulled back with respect
to the section $\sigma_{\lambda} : \mathrm{SU}(2) \to \mathcal{SO}_0\mathrm{SU}(2)$
corresponding to $(\xi,\rho,\lambda\kappa)$ -- becomes
\[
\sigma^*_{\lambda}\theta_{\lambda}= \begin{pmatrix} 0 & -(\lambda^2+2)\kappa & -\lambda\rho \\ (\lambda^2+2)\kappa & 0 & \lambda\xi \\  -\lambda\rho & \lambda\xi & 0 \end{pmatrix}.
\]
Using the normalisation of \cref{Appendix}, we compute for the Chern--Simons $3$-form
\[
\sigma^*_{\lambda}\CS(\theta_{\lambda})=\left(\frac{8}{16\pi^2}\right)(\lambda^4+2\lambda^2+2)\,\xi\wedge\rho\wedge\kappa.
\]
Consequently, \eqref{eqn:volumethresphere} gives
\[
\int_{\mathrm{SU}(2)}\sigma^*_{\lambda}\CS(\theta_{\lambda})=\lambda^4+2\lambda^2+2>0,
\]
so that \cref{prop:Lorentz-flat-ext} implies that there is no isometric immersion of $g_{\lambda}$ into $\R^{2,2}$ for all $\lambda \in \R\setminus\{0\}$ and that there is no isometric immersion of $g_{\lambda}$ into $\R^{3,1}$ for all $\lambda \in \R\setminus\{0\}$ such that $\lambda^4+2\lambda^2+2 \notin \Z$.

\subsection{Equiaffine immersion}

Real projective $3$-space $\mathbb{RP}^3$ equipped with its standard metric $g_{\mathrm{St}}$ is isometric to $\mathrm{SO}(3)$ equipped with the metric $g=\tfrac{1}{4}\left((\omega_1)^2+(\omega_2)^2+\psi^2\right)$. The Levi-Civita connection form $\theta$ of $g$ -- when pulled back with respect to the section $\sigma : \mathrm{SO}(3) \to \mathcal{SO}\mathrm{SO}(3)$ corresponding to $\tfrac{1}{2}(\omega_1,\omega_2,\psi)$ -- becomes
\[
\sigma^*\theta=\frac{1}{2}\begin{pmatrix} 0 & -\psi & \omega_2 \\ \psi & 0 & -\omega_1 \\ -\omega_2 & \omega_1 & 0 \end{pmatrix}.
\]
Using the normalisation of \cref{Appendix}, we compute for the Chern--Simons $3$-form $\sigma^*\CS(\theta)=\tfrac{1}{16\pi^2}\,\omega_1\wedge\omega_2\wedge\psi$ so that
\[
c_{\sigma}=\int_{\mathrm{SO}(3)}\sigma^*\CS(\theta)=\frac{1}{2} \notin \Z.
\]
\cref{prop:Riemann-flat-ext} implies that there is no isometric immersion of $(\mathbb{RP}^3,g_{\mathrm{St}})$ into $\mathbb{E}^4$, as previously observed by Chern--Simons \cite{MR0353327}. We can however say a bit more. Let $\nabla$ denote the Levi-Civita connection of $g_{\mathrm{St}}$ on $T\mathbb{RP}^3$ and $\nu$ the $\nabla$-parallel volume form on $\mathbb{RP}^3$ corresponding to $\omega_1\wedge\omega_2\wedge\psi$. Since the normalisation of $\B{\,\cdot\,}{\cdot\,}$ for $\mathrm{SL}(3,\R)$ is the same as for $\mathrm{SO}(3)$, our \cref{thm:equiaffine} implies:
\begin{prop}
There exists no global equiaffine immersion of $(\mathbb{RP}^3,\nabla,\nu)$ into $\R^4$. 
\end{prop}

\subsection*{Funding} This research was funded in part by the Austrian Science Fund
(FWF): 10.55776/P33559. For open access purposes, the authors have applied a CC BY
public copyright license to any author-accepted manuscript version arising from this
submission.  This article is based upon work from COST Action CaLISTA CA21109
supported by COST (European Cooperation in Science and
Technology). https://www.cost.eu. The authors also would like to thank the Isaac
Newton Institute for Mathematical Sciences, Cambridge, for support and hospitality
during the programme Twistor theory, where work on this paper was undertaken. This
work was supported by EPSRC grant EP/Z000580/1. T.M. was partially supported by the
DFG priority programme ``Geometry at infinity'' SPP 2026: ME 4899/1-2.

\subsection*{Data availability statement} Data availability is not applicable to this article 
as no new data were created or analyzed in this study.

\subsection*{Conflict of interest} On behalf of all authors, the corresponding author
states that there is no conflict of interest. 

\providecommand{\mr}[1]{\href{http://www.ams.org/mathscinet-getitem?mr=#1}{MR~#1}}
\providecommand{\zbl}[1]{\href{http://www.zentralblatt-math.org/zmath/en/search/?q=an:#1}{zbM~#1}}
\providecommand{\arxiv}[1]{\href{http://www.arxiv.org/abs/#1}{arXiv:#1 }}
\providecommand{\doi}[1]{\href{http://dx.doi.org/#1}{DOI~#1}}
\providecommand{\href}[2]{#2}


\begin{thebibliography}{{\v {C}ap}S09}

\bibitem[BE88]{MR0936085}
\bgroup\scshape{}D.~M. Burns, Jr.\egroup{}, \bgroup\scshape{}C.~L.
  Epstein\egroup{}, A global invariant for three-dimensional {CR}-manifolds,
  \emph{Invent. Math.} \textbf{92} no.~2 (1988), 333--348.

\bibitem[{\v {C}ap}S09]{MR2532439}
\bgroup\scshape{}A.~{\v {C}ap}\egroup{},
  \bgroup\scshape{}J.~Slov{\'a}k\egroup{}, \emph{Parabolic geometries. I},
  \emph{Mathematical Surveys and Monographs} \textbf{154}, American
  Mathematical Society, Providence, RI, 2009, Background and general theory.

\bibitem[Car32]{MR1556687}
\bgroup\scshape{}E.~Cartan\egroup{}, Sur la g\'eom\'etrie pseudo-conforme des
  hypersurfaces de l'espace de deux variables complexes {II},  \emph{Ann.
  Scuola Norm. Super. Pisa Cl. Sci. (2)} \textbf{1} no.~4 (1932), 333--354.

\bibitem[CS74]{MR0353327}
\bgroup\scshape{}S.~S. Chern\egroup{}, \bgroup\scshape{}J.~Simons\egroup{},
  Characteristic forms and geometric invariants,  \emph{Ann. of Math. (2)}
  \textbf{99} (1974), 48--69.

\bibitem[Fox05]{MR2189678}
\bgroup\scshape{}D.~J.~F. Fox\egroup{}, Contact projective structures,
  \emph{Indiana Univ. Math. J.} \textbf{54} no.~6 (2005), 1547--1598.

\bibitem[Fre95]{MR1337109}
\bgroup\scshape{}D.~S. Freed\egroup{}, Classical {C}hern-{S}imons theory. {I},
  \emph{Adv. Math.} \textbf{113} no.~2 (1995), 237--303.

\bibitem[FH91]{MR1153249}
\bgroup\scshape{}W.~Fulton\egroup{}, \bgroup\scshape{}J.~Harris\egroup{},
  \emph{Representation theory}, \emph{Graduate Texts in Mathematics}
  \textbf{129}, Springer-Verlag, New York, 1991, A first course, Readings in
  Mathematics.

\bibitem[GM14]{MR3159164}
\bgroup\scshape{}C.~Guillarmou\egroup{}, \bgroup\scshape{}S.~Moroianu\egroup{},
  Chern-{S}imons line bundle on {T}eichm\"uller space,  \emph{Geom. Topol.}
  \textbf{18} no.~1 (2014), 327--377.

\bibitem[MP14]{MR3228423}
\bgroup\scshape{}A.~McIntyre\egroup{}, \bgroup\scshape{}J.~Park\egroup{}, Tau
  function and {C}hern-{S}imons invariant,  \emph{Adv. Math.} \textbf{262}
  (2014), 1--58.

\bibitem[NS94]{MR1311248}
\bgroup\scshape{}K.~Nomizu\egroup{}, \bgroup\scshape{}T.~Sasaki\egroup{},
  \emph{Affine differential geometry}, \emph{Cambridge Tracts in Mathematics}
  \textbf{111}, Cambridge University Press, Cambridge, 1994, Geometry of affine
  immersions.

\bibitem[Yos85]{MR807069}
\bgroup\scshape{}T.~Yoshida\egroup{}, The {$\eta$}-invariant of hyperbolic
  {$3$}-manifolds,  \emph{Invent. Math.} \textbf{81} no.~3 (1985), 473--514.

\end{thebibliography}
\end{document}